\documentclass[11pt]{amsart}
\usepackage{amscd}
\usepackage{comment}
\usepackage{amsmath, amssymb, amsgen, amsthm, amscd, xspace, color}
\newcommand{\pf}{\begin{proof}}
\newcommand{\epf}{\end{proof}}
\newcommand{\eq}{\begin{equation}}
\newcommand{\eeq}{\end{equation}}
\newcommand{\eqn}{\begin{equation*}}
\newcommand{\eeqn}{\end{equation*}}

\newtheorem{theorem}[equation]{Theorem}

\newtheorem{lemma}[equation]{Lemma}

\theoremstyle{remark}

\theoremstyle{definition}

\numberwithin{equation}{section}
\allowdisplaybreaks[1]
\begin{document}

\title{The bounded spherical functions on the Cartan Motion group and generalizations for the Eigenspaces of the Laplacian on $\mathbb{R}^{n}$}
\author{Jingzhe Xu}
\address[Xu]{Department of Mathematics, Hong Kong University of Science
and technology,
Clear Water Bay, Kowloon, Hong Kong SAR, China}
\email{jxuad@ust.hk}
\abstract{The bounded spherical functions are determined for a real Cartan Motion group which is a generalization for
the case when the Cartan Motion group is complex written by Sigurdur Helgason[1]. Also, I will do a further step of the Laplacian on
$\mathbb{R}^{n}$. I consider the case when $K$ is transitive on the spheres about 0 in $\mathbb{R}^{n}$,$n>1$}
\endabstract

%
\maketitle     
%
\section{Introduction}
Consider a symmetric space $X=G/K$ of noncompact type, $G$ being a connected noncompact semisimple lie group with finite center and $K$
a maximal compact subgroup. Let $g=k+p$ be the corresponding Cartan decomposition, $p$ being the orthocomplement of $k$ relative to the
killing form    of $g$. Let $a\subset p$ be a maximal abelian subspace. Let $G_{0}$ be the Cartan Motion group. This
group is defined as the semidirect product of $K$ and $p$ with respect to the adjoint action of $K$ on $p$. The $X_{0}=G_{0}/K$ is naturally
identified with the Euclidean space $p$. The element $g_{0}=(k,Y)$ actions on $p$ by

$g_{0}(Y^{'})=Ad(k)Y^{'}+Y \ k\in K,Y,Y^{'}\in p$,

So the algebra $\mathbb{D}(X_{0})$ of $G_{0}$-invariant differential operators on $X_{0}$ is identified with the algebra of $Ad(K)$-invariant
constant coefficient differential operators on $p$. The corresponding spherical functions on $X_{0}$ are given by
      $\psi_{\lambda}(Y)=\int_{k}e^{i\lambda(Ad(k)Y)}dk  \    \lambda\in a_{c}^{*}$

and $\psi_{\lambda}=\psi_{\mu}$ if and only if $\lambda$ and $\mu$ are $W$-conjugate. See e.g.[2],IV\S 4. Again, the maximal ideal space of
$L^{\natural}(G_{0})$ is up to $W$-invariance identified with the set of $\lambda$ in $a_{c}^{*}$ for which $\psi_{\lambda}$ is bounded. Since
$\rho$ is relative to the curvature of $G/K$ it is natural to expect the bounded $\psi_{\lambda}$ to come from replacing $c(\rho)$ by the origin,
where $c(\rho)$ is for the semisimple case also proved by Sigurdur Helgason [3]. In the words, $\psi_{\lambda}$ is would be expected to be bounded
if and only if $\lambda$ is real, that is $\lambda\in a^{*}$. In [1], Sigurdur Helgason proved when $G$ is complex, the spherical function $\psi_{\lambda}$ on $G_{0}$ is bounded if and only if $\lambda$ is real, i.e. $\lambda\in a^{*}$ mainly by using two results proved by Harish-Chandra
[4] and [5]. In this paper, I use a different way to prove when $G$ is real, the spherical function $\psi_{\lambda}$ on $G_{0}$ is bounded also
if and only if $\lambda$ is real. In this way, we generalize Sigurdur Helgason's results.

In Sigurdur Helgason's another paper [6], he considered Eigenspaces of the Laplacian on $\mathbb{R}^{n}$. Let $L$ denote the usual Laplacian on $\mathbb{R}^{n}$ and for each $\lambda \in \mathbb{C}$ let $\mathcal{E}_{\lambda}(\mathbb{R}^{n})$ denote the eigenspace
$\mathcal{E}_{\lambda}(\mathbb{R}^{n})=\{f\in \mathcal{E}(\mathbb{R}^{n})\mid Lf=-\lambda^{2}f\}$ with the topology induced by that of $\mathcal{E}(\mathbb{R}^{n})$. Let $G$ denote the group of all isometries of $\mathbb{R}^{n}$, and $K$ the group of rotations $O(n)$. Sigurdur Helgason mainly proved the natural action of $G$ on $\mathcal{E}_{\lambda}(\mathbb{R}^{n})$ is irreducible if and only if $\lambda\neq 0$. I will
prove when $K\subset O(n)$ is transitive on the spheres about 0 in $\mathbb{R}^{n}$,$n>1$ instead of $O(n)$ and $G=K\rtimes \mathbb{R}^{n}$, the same results holds. In this way, we do a further step of this kind of problem. Meanwhile, I will specific all the groups $K$ which is transitive on the spheres about 0 in $\mathbb{R}^{n}$.

Finally, according to [7], we know when $K\subset O(n)$ is transitively on the spheres about 0, the specific form of the spherical functions on
$K\rtimes \mathbb{R}^{n}/K\cong \mathbb{R}^{n}$. Then I will give a estimation for it when $r\rightarrow \infty$.

%
\section{The main theorem}
The notion of induced spherical function mirrors the notion of induced representation. Let $Q\subset G$ be a closed subgroup such that $K$ is transitive on $G/Q$, i.e.$G=KQ$,i.e.$G=QK$,i.e. $Q$ is transitive on $G/K$. Let $\zeta: Q\rightarrow \mathbb{C}$ be spherical for $(Q,Q\bigcap K)$.
The induced spherical function is
$[Ind_{Q}^{G}(\zeta)](g)=\int_{K}\widetilde{\zeta}(gk)d\mu_{K}(k)$ \  where $\widetilde{\zeta}(kq)=\zeta(q)\Delta_{G/Q}(q)^{-\frac{1}{2}}$

Here $\Delta_{G/Q}:Q\rightarrow \mathbb{R}^{n}$ is the quotient of modular functions, $\Delta_{G/Q}(q)=\Delta_{G}(q)/\Delta_{Q}(q)=\Delta_{Q}(q)^{-1}$

\begin{theorem}\label{equal}
Let $\lambda\in a_{c}^{*}$. Then $\psi_{\lambda}(Y)$ is the induced spherical function $Ind_{p}^{G}(\varphi_{\lambda})$, where
$\varphi_{\lambda}(Y)=e^{i\lambda(Y)}$ for every $Y\in a\subset p$.
\end{theorem}
\begin{proof}
Apply above formula to $\varphi_{\lambda}$ with $Q=P$. Since $G=K\rtimes p$ and $p$ are unimodular, it says that the induced spherical function is given by
$Ind_{p}^{G}(\varphi_{\lambda})(k,Y)=\int_{K}\varphi_{\lambda}(Ad(k_{0})Y)dk_{0}=\int_{K}e^{i\lambda(Ad(k_{0})Y)}dk_{0}=\int_{K}e^{i\lambda(Ad(k)Y)}dk
=\psi_{\lambda}(Y)=\psi_{\lambda}(k,Y)$ \ for\ $Y\in p, k\in K$. Note that $(k_{0},0)(k,Y)=(k_{0}k,Ad(k_{0})Y)$.
\end{proof}

We apply the Mackey little group method to $G$ relative to its normal subgroup $p$:If $\psi$ is an irreducible unitary representation of $G$,then it can be constructed (up to unitary equivalence) as follows: If $\varphi_{\lambda}(Y)=e^{i\lambda(Y)}$, where $\lambda \in p^{*}$, $Y\in p$, let
$K_{\varphi_{\lambda}}=\{k\in K\mid \varphi_{\lambda}(Ad(k)Y)=\varphi_{\lambda}(Y) \ \forall Y\in p\}$. $K_{\varphi_{\lambda}}$ is a closed subgroup of $K$. Let $G_{\varphi_{\lambda}}=K_{\varphi_{\lambda}}\rtimes p$. Write $\widetilde{\varphi_{\lambda}}$ for the extension of $\varphi_{\lambda}$ to
$G_{\varphi_{\lambda}}$ given by $\widetilde{\varphi_{\lambda}}((k,Y))=\varphi_{\lambda}(Y)$. If $\gamma$ is an irreducible unitary representation of $K_{\varphi_{\lambda}}$, let $\widetilde{\gamma}$ denote its extension of $G_{\varphi_{\lambda}}$ given by
$\widetilde{\gamma}((k,Y))=\gamma(k)$. Denote $\psi_{\varphi_{\lambda},\gamma}=Ind_{G_{\varphi_{\lambda}}}^{G}(\widetilde{\varphi_{\lambda}}\otimes \widetilde{\gamma})$, then there exist choices of $\varphi_{\lambda}$ and $\gamma$ such that $\psi=\psi_{\varphi_{\lambda},\gamma}$.

\begin{theorem}\label{equal}
In the notation above, $\psi_{\varphi_{\lambda},\gamma}$ has a $K$-fixed vector is given(up to scalar multiple by $u((k,Y))=e^{-i\lambda(Ad(k^{-1})Y)}$, if $\varphi_{\lambda}=e^{i\lambda(Y)}$.
\begin{proof}
The representation space $H_{\psi}$ of $\psi=\psi_{\varphi_{\lambda},\gamma}$ consists of all $L^{2}$ functions $f:G\rightarrow H_{\psi}$
such that $f(g^{'}(k^{'},x^{'}))=\gamma(k^{'})^{-1}\varphi_{\lambda}(x^{'})^{-1}f(g^{'})$ for $g^{'}\in G$, $x^{'}\in p$, $k^{'}\in K_{\varphi_{\lambda}}$, and $\psi$ acts by $(\psi(g)f)(g^{'})=f(g^{-1}g^{'})$.

Now suppose that $0\neq f\neq \in H_{\psi}$ is fixed under $\psi(K)$. If $k^{'}\in K_{\varphi_{\lambda}}$, then $\gamma(k^{'})f(1)=f(1)$. If $f(1)=0$, then $f(G_{\varphi_{\lambda}})=0$ and $K$-invariance says $f=0$, contrary to the assumption. Thus $f(1)\neq 0$ and irreducibility of $\gamma$ forces $\gamma$ to be trivial.

Conversely, if $\gamma$ is trivial, then $u((k,Y))=e^{-i\lambda(Ad(k^{-1})Y)}$ is a nonzero $K$-fixed vector in $H_{\psi}$. And it is the only one, up to scalar multiple, because any two $K$-fixed vectors must be proportional.
\end{proof}
\end{theorem}
\begin{lemma}\label{twist}
In the notation above, $Ind_{G_{\varphi_{\lambda}}}^{G}(\widetilde{\varphi_{\lambda}})$ is unitary equivalent to the subrepresentation of $Ind_{p}^{G}(\widetilde{\varphi_{\lambda}})$ generated by the $K$-fixed unit vector $u((k,Y))=e^{-i\lambda(Ad(k^{-1})Y)}$, if $\varphi_{\lambda}=e^{i\lambda(Y)}$.
\end{lemma}

\begin{theorem}\label{equal}
Let $\varphi$ be a $(K\rtimes p, K)$-spheirical function. Then $\varphi$ is positive definite if and only if it is of the form $\psi_{\lambda}$ for some $\lambda\in a^{*}$. Further, if $\lambda,\lambda^{'}\in a^{*}$, then $\psi_{\lambda}=\psi_{\lambda^{'}}$ if and only if $\lambda^{'}\in Ad(k)\lambda$.
\end{theorem}
\begin{proof}
Let $\lambda\in a^{*}$. $\sum_{i,j}e^{i\lambda(Ad(k)(-Y_{j}+Y_{i}))}\overline{c_{j}}c_{i}=\sum_{i,j}(e^{-i\lambda(Ad(k)Y_{j})}\overline{c_{j}})e^{i\lambda(Ad(k)Y_{i})}c_{i}
=(\sum_{i}e^{i\lambda(Ad(k)Y_{i})}c_{i})\overline{(\sum_{j}e^{i\lambda(Ad(k)Y_{j})}c_{j})}\geqslant 0$.

Since $\psi_{\lambda}$ is a limit of non-negative linear combinations of positive definite functions on $\mathbb{R}^{n}$, so it is positive definite.

Now let $\varphi$ be a positive definite $(G,K)$-spherical function. Let $\Pi_{\varphi}$ be the associated irreducible unitary representation, and $H_{\varphi}$ the representation space, such that there is a $K$-fixed unit vector $u_{\varphi}\in H_{\varphi}$ and let $\varphi(g)=<u_{\varphi},\prod_{\varphi}(g)u_{\varphi}>$ for all $g\in G$. Following the discussion of the Mackey little group method, and Theorem 2.2 , we have  $\varphi_{\lambda}(Y)$ for some $\lambda\in a^{*}, Y\in p$, s.t. $\Pi_{\varphi}$ is unitarily equivalent to $Ind_{G_{\varphi_{\lambda}}}^{G}(\widetilde{\varphi_{\lambda}})$. Making the identification one, $K$-fixed unit vector in $H_{\varphi}$ is given by
$u((k,Y))=e^{-i\lambda(Ad(k^{-1})Y)}$. We have $\lambda\in a{*}$ s.t.$\varphi_{\lambda}=e^{i\lambda(Y)}$ and from above several Theorems and Lemma, we compute:

$\varphi(Y)=<u,\Pi_{\varphi}(Y)u >$
=$\varphi(Y)=<u,Ind_{G_{\varphi_{\lambda}}}^{G}(\widetilde{\varphi_{\lambda}})(Y)u >$

=$\varphi(Y)=<u,Ind_{G_{p}}^{G}(\varphi_{\lambda})(Y)u >$

=$Ind_{G_{p}}^{G}(\varphi_{\lambda})(Y)$

=$\psi_{\lambda}(Y)$

=$\int_{K}e^{i\lambda(Ad(k)Y)}dk$

For the second, if $\lambda^{'}=Ad(k_{0})\lambda$ for some $k_{0}\in K$, we have:

$\psi_{\lambda^{'}}(Y)=\int_{K}e^{i\lambda^{'}(Ad(k)Y)}dk=\int_{K}e^{iAd(k_{0})\lambda(Ad(k)Y)}dk$
=$\int_{K}e^{i\lambda(Ad(k_{0}^{-1})Ad(k)Y)}dk=\int_{K}e^{i\lambda(Ad(k)Y)}dk=\psi_{\lambda}(Y)$.

Conversely, suppose that $\lambda^{'},\lambda\in a^{*}$ with $\psi_{\lambda^{'}}=\psi_{\lambda}$. Then
(up to unitary equivalence)$Ind_{G_{\varphi_{\lambda}}}^{G}(\widetilde{\varphi_{\lambda}})$=
$Ind_{G_{\varphi_{\lambda^{'}}}}^{G}(\widetilde{\varphi_{\lambda^{'}}})$. That gives us direct integral decompositions

$\int_{K}^{\bigoplus}\psi_{Ad(k)\lambda}dk=Ind_{G_{\varphi_{\lambda}}}^{G}(\widetilde{\varphi_{\lambda}})\mid p$
=$Ind_{G_{\varphi_{\lambda^{'}}}}^{G}(\widetilde{\varphi_{\lambda^{'}}})\mid p=\int_{K}^{\bigoplus}\psi_{Ad(k)\lambda^{'}}dk$
\end{proof}

\begin{theorem}\label{equal}
If $N$ is an $n$-step group with $n\geq 3$, then there are no Gelfand pairs $(K,N)$, where $K\in Aut(N)$.
\end{theorem}

\begin{theorem}\label{equal}
We first consider $K$-spherical functions associated to a Gelfand pair $(K,N)$.

Suppose $\phi$ is a bounded $K$-spherical function on $N$. Then there is a $\pi\in \hat{N}$ and a unit vector $\xi\in H_{\pi}$ such that
\begin{equation}\label{l-invariant elements}
\phi(x)=\int_{K}<\pi(k.x)\xi,\xi>dk
\end{equation}
for each $x\in N$
\end{theorem}
\begin{proof}
Let $\lambda_{\phi}:L_{K}^{1}(N)\rightarrow \mathbb{C}$ be given by integration against $\phi$.

Since $L^{1}(N)$ is a symmetric Banach *-algebra,[8],there is a representation $\bar{\pi}$ of  $L^{1}(N)$ and a one-dimensional subspace $H_{\phi}$ of $H_{\bar{\pi}}$ such that $(\bar{\pi}\mid_{L_{K}^{1}(N)},H_{\phi})$ is equivalent to $(\lambda_{\phi},\mathbb{C})$. As $\lambda_{\phi}$ is irreducible, the extension $\bar{\pi}$ is also irreducible(cf.[9]). Using approximate identities at each point of $N$, one can show that $\bar{\pi}$ is the integrated version of some $\pi\in \hat{N}$, with $H_{\pi}=H_{\bar{\pi}}$.

Choose $\xi\in H_{\phi}$ with $\left \|\xi  \right \|=1$. Then for each $f\in L_{K}^{1}(N)$, $\pi(f)\xi=\lambda_{\phi}(f)\xi$, so that
\begin{equation}\label{l-invariant elements}
\begin{split}
&<\phi,f >=\lambda_{\phi}(f)=<\pi(f)\xi,\xi>\\
&=\int_{N}f(x)<\pi(x)\xi,\xi>dx\\
&=\int_{K}\int_{N}f(k^{-1}.x)<\pi(x)\xi,\xi>dxdk\\
\end{split}
\end{equation}
since $f$ is $K$-invariant
\begin{equation}\label{l-invariant elements}
=\int_{K}\int_{N}f(k.x)<\pi(x)\xi,\xi>dxdk
\end{equation}
Since $\phi$ is $K$-invariant, we change the order of integration and obtain
\begin{equation}\label{l-invariant elements}
\phi(x)=\int_{K}<\pi(x)\xi,\xi>dk
\end{equation}
\end{proof}

A complex-valued continous function $\phi$ on a locally compact group $G$ is called positive definite if
$\sum_{i,j=1}^{n}\phi(x_{i}^{-1}x_{j})\alpha_{i}\overline{\alpha_{j}}\geq 0$ for all finite sets
$x_{1},\ldots ,x_{n}$ of elements in $G$ and any complex numbers $\alpha_{1},\ldots ,\alpha_{n}$.

\begin{theorem}\label{equal}
For Gelfand pair $(K,N)$, where $N$ is  at most 2-step nilpotent lie group, if $\phi$ is a bounded $K$-spherical function on $N$ is and only if  $\phi$ is positive definite.
\end{theorem}
\begin{proof}
If $\phi$ is a bounded $K$-spherical function on $N$, for all finite sets
$x_{1},\ldots ,x_{n}$ of elements in $G$ and any complex numbers $\alpha_{1},\ldots ,\alpha_{n}$, we have:

$\sum_{i,j=1}^{n}\phi(x_{i}^{-1}x_{j})=\int_{K}<\pi(k(x_{i}^{-1}x_{j})\xi,\xi>dk$
=$\int_{K}<\pi(k(x_{i})^{-1}k(x_{j})\xi,\xi>dk=\int_{K}<\pi(k(x_{j})\xi,k(x_{i})\xi>dk$.

Therefore,
$\sum_{i,j=1}^{n}\phi(x_{i}^{-1}x_{j})\alpha_{i}\overline{\alpha_{j}}$
=$\sum_{i,j=1}^{n}\int_{K}<\overline{\alpha_{j}}\pi(k(x_{j}))\xi,\overline{\alpha_{i}}\pi(k(x_{i}))\xi>dk$
=$\int_{K}<\sum_{j}\overline{\alpha_{j}}\pi(k(x_{j}))\xi,\sum_{i}\overline{\alpha_{i}}\pi(k(x_{i}))\xi>dk\geq 0$.

Conversely, if $\phi$ is a bounded $K$-spherical function on $N$ then $\phi$ is positive definite. Let $\varphi$ be a positive definite $(G,K)$-spherical function. Let $\Pi_{\varphi}$ be the associated irreducible unitary representation, and $H_{\varphi}$ the representation space, such that there is a $K$-fixed unit vector $u_{\varphi}\in H_{\varphi}$ and let $\varphi(g)=<u_{\varphi},\prod_{\varphi}(g)u_{\varphi}>$ for all $g\in G$. Then we have
$\left |\varphi(g)  \right |=\left |<u_{\varphi},\prod_{\varphi}(g)u_{\varphi}>  \right |$
$\leq \left |<u_{\varphi},u_{\varphi}>  \right |^{\frac{1}{2}}\times \left |<\prod_{\varphi}(g)u_{\varphi},\prod_{\varphi}(g)u_{\varphi}>  \right |^{\frac{1}{2}}$
$\leq \left |<u_{\varphi},u_{\varphi}>  \right |=1$

Therefore, $\varphi$ is bounded.
\end{proof}

\begin{theorem}\label{equal}
In the notation just above, assume the group $G$ real . The spherical function
$\psi_{\lambda}$ on $G_{0}$ is bounded if and only if $\lambda$ is real,i.e.$\lambda\in a^{*}$.
\end{theorem}
\begin{proof}
According to Theorem 2.4, we obtain The spherical function $\psi_{\lambda}$ on $G_{0}$ is positive definite if and only if $\lambda$ is real,i.e.$\lambda\in a^{*}$. According to Theorem 2.11, since $p$ is abelian, we know that $\psi_{\lambda}$ is a bounded $K$-spherical function on $p$ is and only if  $\psi_{\lambda}$ is positive definite. Therefore, $\psi_{\lambda}$ on $G_{0}$ is bounded if and only if $\lambda$ is real,i.e.$\lambda\in a^{*}$.
\end{proof}

%
\section{Generalizations for the Eigenspaces of the Laplacian on $\mathbb{R}^{n}$}
\begin{lemma}\label{twist}
[7] Let $K$ be any closed subgroup of $O(n)$, if $K$ is transitive on the spheres about 0 in $\mathbb{R}^{n}$, then $\mathcal{D}(G/K)=\mathbb{C}[\bigtriangleup]$, algebra of polynomials in the Laplace-Beltrami operator $\bigtriangleup=-\sum \partial^{2}/\partial x_{i}^{2}.$
\end{lemma}
\begin{proof}
It is clear that $\mathbb{C}[\bigtriangleup]\subset \mathcal{D}(G/K)$. Now let $D\in \mathcal{D}(G/K)$ be of order m. Then the mth order symbol of $D$ is a pollynomial of pure degree m constant on spheres about 0 in $\mathbb{R}^{n}$, in other words a multiple $cr^{m}$ with $m$ even and $r^{2}=\sum x_{i}^{2}$. Now $D-(c(-\bigtriangleup)^{m/2})\in \mathcal{D}(G/K)$ and $D-(c(-\bigtriangleup)^{m/2})$ has order $<m$. By induction on the order, $D-(c(-\bigtriangleup)^{m/2})\in \mathcal{D}(G/K)$, so we have $D\in \mathcal{D}(G/K)$.
\end{proof}
Let $L$ denote the usual Laplacian on $\mathbb{R}^{n}$ and for each $\lambda\in \mathbb{C}$ let $\mathcal{E}_{\lambda}(\mathbb{R}^{n})$ denote the eigenspace

\begin{equation}\label{l-invariant elements}
\mathcal{E}_{\lambda}(\mathbb{R}^{n})=\{f\in \mathcal{E}(\mathbb{R}^{n})\mid Lf=-\lambda^{2}f\}
\end{equation}

with the topogy induced by that of $\mathcal{E}(\mathbb{R}^{n})$. Let $G=K\rtimes \mathbb{R}^{n}$ and $K$ is the closed subgroup of $O(n)$ as well as acting transitive on the spheres about 0 in $\mathbb{R}^{n}$.

\begin{theorem}\label{equal}
The natural action of $G$ on $\mathcal{E}_{\lambda}(\mathbb{R}^{n})$ is irreducible if and only if $\lambda\neq 0$.
\end{theorem}
\begin{proof}
It is clear that each function
\begin{equation}\label{l-invariant elements}
f(x)=\int_{S^{n-1}}e^{i\lambda(x,w)}F(w)dw,    \    F\in L^{2}(S^{n-1}),
\end{equation}
lies in $\mathcal{E}_{\lambda}(\mathbb{R}^{n})$; here $(,)$ denotes the usual inner product on $\mathbb{R}^{n}$ and $dw$ the normalized volume element.
\end{proof}

\begin{lemma}\label{twist}
Let $\lambda\neq 0$. Then the mapping $F\rightarrow f$ defined by (3.4) is one-to-one.
\end{lemma}
\begin{proof}
Let $p(\zeta)=p(\zeta_{1},\cdots \zeta_{n})$ be a polynomial and $D$ the corresponding constant coefficient differential operator on $\mathbb{R}^{n}$ such that
\begin{equation}\label{l-invariant elements}
\int_{\mathbb{R}^{n}}e^{i(x,\zeta)}D_{x}(e^{-(\frac{1}{2})\left |x  \right |^{2}})=p(\zeta)e^{-(\frac{1}{2})(\zeta_{1}^{2}+\cdots +\zeta_{n}^{2})}
\end{equation}
for$\zeta\in \mathbb{C}^{n}$. If $f\equiv=0$ in (1) we deduce from (3.6) that
\begin{equation}\label{l-invariant elements}
\int_{S^{n-1}}p(\lambda w_{1},\cdots ,\lambda w_{n})F(w)dw=0
\end{equation}
Since $\lambda \neq 0$, this implies $F\equiv 0$.
\end{proof}

\begin{lemma}\label{twist}
Let $\lambda\neq 0$. The $K$-finite solutions $f$ of the equation $Lf=-\lambda^{2}f$ are precisely
\begin{equation}\label{l-invariant elements}
f(x)=\int_{S^{n-1}}e^{i\lambda(x,w)}F(w)dw
\end{equation}
where $F$ is a $K$-finite function on $S^{n-1}$.
\end{lemma}
\begin{proof}
Let $\delta$ be an irreducible representation of $K$ and if $\Sigma$ is any sphere in $\mathbb{R}^{n}$ with center at 0 let $\mathcal{E}_{\delta}(\Sigma)$ denote the space of $K$-finite functions in $\mathcal{E}(\Sigma)$ of type $\delta$. We know from Lemma 1.5 p.134 in [10] that if $\Sigma$ is suitably chosen each function $f\mid \Sigma$ to$\Sigma$. With $F$ and $f$ as in (3.4) it follows that the maps

$F\rightarrow f\mid \Sigma,  \   F\rightarrow f  \  F\in L^{2}(S^{n-1})$

are one-to-one and commute with the action of $K$. For reasons of dimensionality, the first must therefore map $\mathcal{E}_{\delta}(S^{n-1})$ onto $\mathcal{E}_{\delta}(\Sigma)$. The lemma now follows.
\end{proof}

For $\lambda\neq 0$ let $\mathcal{H}_{\lambda}$ denote the space of functions $f$ as defined in (3.4); $\mathcal{H}_{\lambda}$ is a Hilbert space if the norm of $f$ is the $L^{2}$ norm of $F$ on $S^{n-1}$.

\begin{lemma}\label{twist}
Let $\lambda\neq 0$. Then the space $\mathcal{H}_{\lambda}$ is dense in $\mathcal{E}_{\lambda}(\mathbb{R}^{n})$.
\end{lemma}
\begin{proof}
Each eigenfunction of $L$ can be expanded in a convergent series of $K$-finite eigenfunctions (cf. Sect. 5 [6])
so the lemma follows from Lemma 3.8.
\end{proof}
We can now prove Theorem 3.3. We first prove that $G$ acts irreducibly on $\mathcal{H}_{\lambda}$. Let $V\neq 0$ be a closed invariant subspace of $\mathcal{H}_{\lambda}$. Then there exists an $h\in V$ such that $h(0)=1$. We write
\begin{equation}\label{l-invariant elements}
h(x)=\int_{S^{n-1}}e^{i\lambda(x,w)}H(w)dw
\end{equation}
and the average $h^{\natural}(x)=\int_{K}h(k.x)dk$ is then
\begin{equation}\label{l-invariant elements}
h^{\natural}(x)=\varphi_{\lambda}(x)=\int_{S^{n-1}}e^{i\lambda(x,w)}dw.
\end{equation}
If $f$ in (3.4) lies in the annihilator $V^{0}$ of $V$ the functions $F$ and $H$ are orthogonal on $S^{n-1}$. Since $V^{0}$ is $K$-invariant this remains true for $H$ replaced by its integral over $K$, in other words $\varphi_{\lambda}$ belongs to the double annihilator $(V^{0})^{0}=V$. Now, since $V$ is invariant under translations it follows that for each $t\in \mathbb{R}$ the function
\begin{equation}\label{l-invariant elements}
x\rightarrow \int_{S^{n-1}}e^{i\lambda(x,w)}e^{i\lambda(t,w)}dw
\end{equation}
belongs to $V$. But then Lemma 3.5 shows that the annihilator of $V$ in $\mathcal{H}_{\lambda}$ is ${0}$, whence the irreducibility of $G$ on $\mathcal{H}_{\lambda}$.

Passing now to $\mathcal{E}_{\lambda}$ let $V\subset \mathcal{E}_{\lambda}$ be a closed invariant subspace. Then $V\cap \mathcal{H}_{\lambda}$ is an invariant subspace of  $\mathcal{H}_{\lambda}$; Schwartz' inequality shows easily that it is closed. Thus, by the above, $V\subset \mathcal{E}_{\lambda}$ is $\{0\}$ or $\mathcal{H}_{\lambda}$. In the second case $V=\mathcal{E}_{\lambda}$ by Lemma 3.10. In the first case consider for each $f\in V$ the convergent expansion
\begin{equation}\label{l-invariant elements}
f=\sum_{\delta\in \hat{K}}\alpha_{\delta}*f
\end{equation}
where $\alpha_{\delta}=d(\delta)\chi_{\delta}^{*}$ and
\begin{equation}\label{l-invariant elements}
(\alpha_{\delta}*f)(x)=\int_{K}\alpha_{\delta}(k)f(k^{-1}.x)dk
\end{equation}
$\chi_{\delta}$ being the character of $\delta$.Then $\alpha_{\delta}*f\in \mathcal{H}_{\lambda}$ by Lemma 3.8. Let $V^{0}\subset \mathcal{E}^{'}(\mathbb{R}^{n})$ be the annihilator of $V$. Then $V^{0}$ is $G$-invariant  and if $T\in V^{0}$,
\begin{equation}\label{l-invariant elements}
\int_{\mathbb{R}^{n}}(\alpha_{\delta}*f)(x)dT(x)=\int_{K}\alpha_{\delta}(k)\int_{\mathbb{R}^{n}}f(x)dT(k.x)dk
\end{equation}
so $\alpha_{\delta}*f$ belongs to the double annihilator $(V^{0})^{0}=V$. Thus $\alpha_{\delta}*f\in V\cap \mathcal{H}_{\lambda}=\{0\}$ so, by (3.6),$f=0$. Thus $V=\{0\}$ so the proof is finished.

\section{A estimation for some spherical functions and the groups K}
From [7], we know if $K$ is transitive on the spheres about 0 in $\mathbb{R}^{n}$, then the spherical function on $K\rtimes \mathbb{R}^{n}\simeq \mathbb{R}^{n}$ is of the form:
$\varphi_{s}(r)=\varphi(r,s)=\int_{S^{n-1}}e^{s(\xi,x)}d\sigma(\xi)$
=$\frac{\Gamma(\frac{n}{2})}{\sqrt{\pi}\Gamma\frac{n-1}{2}}\int_{0}^{\pi}e^{sr\cos \theta}\sin^{n-2}\theta d\theta$.

Where $s$ is a complex number, and $r=\left \|x  \right \|=\sqrt{x_{1}^{2}+\ldots +x_{n}^{2}}$, $S^{n-1}$ is the unit sphere in $\mathbb{R}^{n}$ and
$\sigma$ the normalized surface measure on $S^{n-1}$.

If Re$s$=0 then it follows from above equation that $\left |\varphi_{s}(x)  \right |\leq 1$ for all $x\in G$.

Clearly,$\varphi_{s}=\varphi_{-s}$. We just need to consider the case Re$s\geq 0$.

\begin{theorem}\label{equal}
If Re$s>0$, we have
$\varphi(r,s)\sim \frac{\Gamma(\frac{n}{2})2^{\frac{n-3}{2}}}{\sqrt{\pi}}\frac{e^{sr}}{(sr)^{\frac{n-1}{2}}}$,

when $r\rightarrow \infty$.
\end{theorem}
\begin{proof}
From above equation, we know that, by an elementary substitution,

$\varphi_{s}(r)=\varphi(r,s)$
=$\frac{\Gamma(\frac{n}{2})}{\sqrt{\pi}\Gamma(\frac{n-1}{2})}\int_{-1}^{1}e^{srt}(1-t^{2})^{\frac{n-3}{2}}dt$

and, setting $t=1-\frac{u}{r}$, we obtain

$\varphi(r,s)=\frac{\Gamma(\frac{n}{2})}{\sqrt{\pi}(\frac{n-1}{2})}\frac{e^{sr}}{r^{\frac{n-1}{2}}}\int_{0}^{2r}e^{-su}u^{\frac{n-3}{2}}(2-\frac{u}{r})^{\frac{n-3}{2}}du$

For Re$s>0$, we get, using Lebesgue's dominated convergence theorem,

$\lim_{r\rightarrow \infty}\int_{0}^{2r}e^{-su}u^{\frac{n-3}{2}}(2-\frac{u}{r})^{\frac{n-3}{2}}du=\frac{\Gamma(\frac{n-1}{2})2^{\frac{n-3}{2}}}{s^{\frac{n-1}{2}}}$
\end{proof}

Finally, I will give all the possible $K$, which is transitive on the spheres about 0 on $\mathbb{R}^{n}$, $n>1$. [11]

When $K$ is transitive on the spheres about 0 in $\mathbb{R}^{n}$,$n>1$, its identity compoment $K^{0}$ is also transitive, and $K=K^{0}F$ where $F$ is a finite subgroup of the normalizer $N_{O(n)}(K^{0})$. The possibilities for $K^{0}$ are as follows:

(1)$n>1$ and $K^{0}=SO(n)$,

(2)$n=2m$ and (i)$K^{0}=SU(m)$ or (ii)$U(m)$,

(3)$n=4m$ and (i)$K^{0}=Sp(m)$ or (ii)$Sp(m).U(1)$ or $Sp(m).Sp(1)$,

(4)$n=7$ and $K^{0}$ is the exceptional group $G_{2}$,

(5)$n=8$ and $K^{0}=Spin(7)$, and

(6)$n=16$ and $K^{0}=Spin(9)$.

In case (1), $N_{O(n)}(K^{0})=O(n)$, so the relevant choices for $F$ are $\{I\}$ and $\{I,-I\}$, so $K$ is either $SO(n)$ or $O(n)$.

In case (2)(i),$N_{O(n)}(K^{0})=U(m)\bigcup  \alpha U(m)$ where $\alpha$ is complex conjugation of $\mathbb{C}^{m}$ over $\mathbb{R}^{m}$. The relevant
choices for $F$ are the finite subgroups of $U(1)\bigcup \alpha U(1)$ where $U(1)$ consists of the unitary scalar matrices $e^{ix}I$, $x$ real. Those
are the cyclic groups $\mathbb{Z}_{l}=\{e^{2\pi ik/l}I\}$ of order $l\geq 1$ and the dihedral groups $\mathbb{D}_{l}=\mathbb{Z}_{l}\bigcup \alpha\mathbb{Z}_{l}$, so $K$ is a group $SU(m)\mathbb{Z}_{l}$ or $SU(m)\mathbb{D}_{l}$. In the case (2)(ii) the relevant possibilities for $F$ are $\{I\}$ and $\{\alpha,I\}$, so $K$ is either $U(m)$ or $U(m)\bigcup \alpha U(m)$.

In case (3){i},(3)(iii),(4),(5),(6), $K^{0}$ has no outer automorphism, so we may take $F$ in the centralizer $Z_{O(n)}(K^{0})$. Thus in the case (3)(i), $F$ can be any subgroup of $Sp(1)$, in the other words, a cyclic group $\mathbb{Z}_{l}$ of order $l$, a binary dihedral group $\mathbb{D}_{l}^{*}$ of order $4l$, a binary tetrahedral group $\mathbb{T}^{*}$ of order 24, a binary octahedral group $\mathbb{O}^{*}$ of order 48, or a binary icosahedral group $\mathbb{I}^{*}$ of order 60. Thus $K$ is a group $Sp(m)\mathbb{Z}_{l}$,$Sp(m)\mathbb{D}_{l}^{*}$,$Sp(m)\mathbb{T}^{*}$, $Sp(m)\mathbb{O}^{*}$ or $Sp(m)\mathbb{I}^{*}$. In case (3)(ii) the relevant possibilities for $F$ are $\{I\}$ and $\{\beta,I\}$, where the $U(1)$ factor of $K^{0}$ consists of all quaternion scalar multiplications by complex numbers $e^{ix}$, $x$ is real, as in the case (2), and $\beta$ is quaternion scalar multiplication by $j$.Thus $K$ is either $Sp(m)U(1)$ or $(Sp(m)U(1))\bigcup (Sp(m)U(1))\beta$. In case (3)(iii), $K^{0}$ is its own $O(n)$-centralizer so $F=\{I\}$ and $K=Sp(m)Sp(1)$.

In case (4),(5),(6),$K^{0}$ is absolutely irreducible on $\mathbb{R}^{n}$, so relevant $F$ would have to consist of real scalars. As $G_{2}$ does not contain $-I$ we see that the relevant $F$ for case (4) are $\{I\}$ and $\{I,-I\}$, resulting in $K=G_{2}$ and $K=G_{2}\bigcup (-I)G_{2}$. Both $Spin(7)$ and $Spin(9)$ do contain $-I$, so $F$ is trivial in case (5),(6).That gives $K=Spin(7)$ in case (5) and $K=Spin(9)$ for case (6).

\end{document}